\documentclass[11pt]{amsart}
\usepackage{amsmath}
\usepackage{amsthm}
\usepackage{amsfonts}
\usepackage{amssymb}
\usepackage[all]{xy}
\usepackage{alltt}
\usepackage{enumitem}
\usepackage{stmaryrd}
\usepackage{xspace}
\usepackage{comment}
\usepackage[
colorlinks=true,urlcolor=blue,linkcolor=blue,citecolor=blue
]{hyperref}

\theoremstyle{plain}
\newtheorem{prop}{Proposition}[section]

\newtheorem{thm}[prop]{Theorem}
\theoremstyle{definition}

\newtheorem{lem}[prop]{Lemma}

\newcommand{\Z}{\ensuremath{\mathbf{Z}}}

\newcommand{\N}{\ensuremath{\mathbf{N}}}
\newcommand{\C}{\ensuremath{\mathbf{C}}}

\DeclareMathOperator{\nd}{d}

\DeclareMathOperator{\im}{im}

\newcommand{\n}{\noindent}

\renewcommand{\Z}{\ensuremath{\mathbf{Z}}}
\renewcommand{\C}{\ensuremath{\mathbf{C}}}

\newcommand{\D}{\ensuremath{\mathbf{D}}}

\usepackage[margin=1.25in]{geometry}

\begin{document}  
 
\title{Half of finite abelian groups are unit groups}
\date{\today}

\author{Sunil K. Chebolu} 
\address{
Department of Mathematics, Illinois State University, Normal, Illinois 61761}
\email{schebol@ilstu.edu}

\author{Keir Lockridge} 
\address {Department of Mathematics \\
Gettysburg College \\
Gettysburg, PA 17325, USA}
\email{klockrid@gettysburg.edu}
 

\begin{abstract} A group is called realizable if it is the group of units in a ring with identity. The classification of realizable groups is a difficult open problem---originally posed by László Fuchs---and is an active area of research. Realizable groups seem rare, but their proportion within a fixed class of groups (cyclic, dihedral, finite abelian, etc.)~varies. To quantify this proportion, we introduce the realizable density of a class of finite groups as an analog of natural density for subsets of the natural numbers. The realizable finite cyclic groups and the realizable finite abelian $p$-groups for $p$ odd have been classified; we prove that their realizable densities are 1/4 and 0, respectively. The realizable finite abelian 2-groups---and more generally the realizable finite abelian groups---have not been fully classified, and these special cases appear quite difficult. Nonetheless, we prove that the realizable density of finite abelian 2-groups is 1 and the realizable density of finite abelian groups is 1/2. Our work combines existing classification theorems for realizable groups with tools from analytic number theory.
\end{abstract}
 
\maketitle

\thispagestyle{empty}


\section{Introduction}

László Fuchs asked the following fundamental question in \cite{Fuchs1960}: \emph{which abelian groups can occur as the group of units in a ring?} In this paper, all rings are assumed to have a multiplicative identity. Though Fuchs originally asked this question for abelian groups, many authors have studied it for nonabelian groups as well. We therefore say that any group $G$ is \textbf{realizable} if there exists a ring $R$ such that $G \cong R^\times$, where $R^\times$ denotes the multiplicative group of units in $R$.

The first significant results in this area were published in the 1960s and 1970s. The realizable cyclic groups (and their realizing rings) were classified by Gilmer and by Pearson and Schneider in \cite{gil, PearsonSchneider1970}. Shortly thereafter, Ditor classified the realizable groups of odd order in \cite{Ditor1971}. Since then, many authors have made progress both by studying the realizable groups within
particular classes of groups and by computing the unit groups of particular
classes of rings. See, for example, Davis and Occhipinti \cite{DO14a,D014b},  Chebolu and Lockridge \cite{CLabelian,CL17Dihedral,CLpgroups,CL17HowMany}, and Del Corso and Dvornicich  \cite{CD18a, CD18b}.  Despite these advances, Fuchs' problem remains open. Even a complete classification of realizable finite abelian 2-groups feels out of reach. 

In this paper, we address the following question: in a given class of groups, what \textit{proportion} of the groups are realizable? Realizable groups seem relatively rare, and as far as we know there has been no attempt to quantify their frequency; this paper is a first step in this direction. Our definition of proportion is modeled on the notion of natural density defined for subsets of the natural numbers $\N$. Given a subset $A \subseteq \N$ and a positive real number $x$, define $A(x)$ to be the cardinality of $A \cap [1, x]$. The {\bf natural density} of $A$ is defined by \[\nd(A) = \lim_{x \to \infty} \frac{A(x)}{x}\] (if the limit exists). The upper density $\nd^+(A)$ and the lower density $\nd^-(A)$ are defined to be the limit superior and limit inferior, respectively, of $A(x)/x$ as $x \to \infty$. These quantities always exist since this function is bounded.

Now let $\mathcal{C}$ be a fixed class of finite groups (up to isomorphism). For a positive real number $x$, let $\mathcal{C}(x)$ denote the number of groups in $\mathcal{C}$ of order at most $x$, and let $\mathcal{C}_r(x)$ denote the number of those groups that are realizable. Define the {\bf realizable density} of $\mathcal{C}$ to be the limit \[\delta(\mathcal{C})=  \lim_{x \to \infty} \frac{\mathcal{C}_r( x)}{\mathcal{C}(x)}\] (if the limit exists). Define the upper realizable density $\delta^{+}(\mathcal{C})$ to be the limit superior as $x \to \infty$ of the above function, and define the lower realizable density $\delta^{-}(\mathcal{C})$ to be the limit inferior. The quantity $\delta(\mathcal{C})$ is our notion of proportion for the realizable groups in the class $\mathcal{C}$. We will also sometimes refer to the density of one class $\mathcal{C}$ contained in another class $\mathcal{D}$; by this we just mean the limit \[ \lim_{x\to\infty} \frac{\mathcal{C}(x)}{\mathcal{D}(x)} \] (if the limit exists).

For example, let $\mathcal{D}$ denote the class of finite dihedral groups. In \cite{CL17Dihedral} we proved that the dihedral group $\D_{2n}$ of order $2n$ is the group of units in a ring of positive characteristic if and only if $n \in \{1,2,3,4,6\}$, and $\D_{2n}$ is the group of units in a ring of characteristic zero if and only if $2n$ is congruent to 4 mod 8. We may ignore the groups in the positive characteristic case: since the list is finite, it will not affect the realizable density. Thus, \[ \delta(\mathcal{D}) = \lim_{x \to \infty} \frac{\lfloor (x+4)/8 \rfloor}{\lfloor x/2 \rfloor} = \frac{1}{4}.\] It is interesting to note that this proportion is determined by the groups realizable in characteristic zero. This also turns out to be the case for the class of finite abelian groups.

Define the following families of isomorphism classes of abelian groups:
\begin{itemize}[leftmargin=2cm]
    \item[$\mathcal{Z}\colon$] finite cyclic groups
    
    \item[$\mathcal{A}_p\colon$] finite abelian $p$-groups (where $p$ is a positive prime)
    \item[$\mathcal{I}\colon$] finite abelian indecomposable groups
    \item[$\mathcal{A}\colon$] finite abelian groups.
\end{itemize}
In this paper, we prove the following theorem.

\begin{thm} For the classes defined above, we have:
  \[
      \delta(\mathcal{Z}) = 1/4, \quad\quad\quad
      \delta(\mathcal{A}_p) = \begin{cases} 0 & \text{if $p$ is odd} \\ 1 & \text{if $p = 2$},
      \end{cases}\quad\quad\quad
      \delta(\mathcal{I}) = 0, \quad\quad\quad
      \delta({\mathcal{A}}) = \frac{1}{2}.
  \]
\end{thm}

The computation of $\delta(\mathcal{Z})$ relies on the classification of realizable finite cyclic groups from \cite{PearsonSchneider1970}. For odd primes $p$, the computation of $\delta(\mathcal{A}_p)$ uses Ditor's classification of realizable groups of odd order in \cite{Ditor1971}, which implies that a realizable finite abelian $p$-group must be elementary when $p$ is odd. For the $p=2$ case,  $\C_2 \times H$ is actually realizable for \textit{any} finite abelian group $H$, and this drives the proportion to 1 even though there is no current full classification of the realizable finite abelian $2$-groups. To compute $\delta(\mathcal{I})$, we use our classification of realizable indecomposable abelian groups given in \cite{CLabelian}.

The most substantial result in this paper is that the realizable density of the class of finite abelian groups is $1/2$. Finite abelian groups that occur as the group of units in a finite ring are so rare that their proportion in $\mathcal{A}$ is zero. Thus, $\delta(\mathcal{A})$ is determined by which groups occur as the group of units in a ring of characteristic zero. This is a difficult, open case. Any such group must contain $\C_2$ or $\C_4$ as a direct factor. Every finite abelian group with $\C_2$ as a direct factor is the group of units in a characteristic zero ring, and these groups account for half of the finite abelian groups. We prove that the realizable finite abelian groups with $\C_4$ as a direct factor but not $\C_2$ have density zero in $\mathcal{A}$. This implies $\delta(\mathcal{A}) = 1/2.$

From analytic number theory, we use several major results, including: (1) the prime number theorem; (2) The Landau--Ramanujan Theorem on the natural density of integers representable as a sum of two squares; (3) the Hardy--Ramanujan asymptotic formula for the partition function $p(n)$; (4) Erdős' theorem that the natural density of the image of the Euler $\varphi$-function is zero (and his asymptotic bound on the image's counting function); and (5) asymptotic results of Erdős--Szekeres and Tóth on the number of finite
abelian groups.

\vspace{.1in}

\n \textbf{Acknowledgments:}
This paper stems from a question raised by Papa Sissokho following a master’s thesis defense on Fuchs' problem given by a graduate student working with the first author at Illinois State University (ISU). The authors thank Papa Sissokho for asking the insightful question that motivated this research. The first author also thanks ISU's Office of Research and the College of Arts and Sciences for the opportunity to participate in a Summer Writing Camp at beautiful Lake Bloomington in Hudson, IL, where parts of this paper were written.

\section{Preliminaries}

In this section, we summarize several needed results in analytic number theory.

\subsection{The number of finite abelian groups}

To compute the realizable density of finite abelian groups, we require the precise asymptotic growth rate of $\mathcal{A}(x)$ (the total number of finite abelian groups of order at most $x$). The following theorem of Erd\H{o}s and Szekeres says that $\mathcal{A}(x)$ grows linearly in $x$. Let $\zeta(z)$ denote the Riemann zeta function. We write $f(x) \sim g(x)$ and say $f$ is \textbf{asymptotic to} $g$ if $\lim_{x\to \infty} f(x)/g(x) = 1.$

\begin{thm}[\cite{ErdosSzekeres1935}]
\label{thm:erdos_szekeres}
If $\mathcal{A}$ is the class of finite abelian groups, then
\[ \mathcal{A}(x) \sim C x, \]
where 
\[ C = \prod_{k=2}^\infty \zeta(k) \approx 2.29485. \]
\end{thm}
\n In particular, the average number of abelian groups of order $n$, for $n\le x$, tends to approximately $2.29485$.

We need the next theorem to prove that if $S\subseteq \N$ is a set with natural density 0, then any class of finite abelian groups with orders in $S$ has density zero in $\mathcal{A}$ (see Lemma \ref{reductiontonaturaldensity} below).

\begin{thm}[\cite{toth}] \label{wk} Let $a(n)$ denote the number of isomorphism classes of abelian groups of order $n$. Then,
\[ B(x) = \sum_{n \le x} a(n)^2 \sim H x  \]
for \[H = \prod_q \left(1 - \frac{1}{q}\right) \left( \sum_{k=0}^\infty \frac{p(k)^2}{q^k} \right),\] where the product is over all positive primes $q$ and $p(k)$ is the partition function.
\end{thm} 

\begin{lem} \label{reductiontonaturaldensity}
 Let $S$ be a set of natural numbers and let $\mathcal{S}$ be any collection of finite abelian groups whose orders lie in $S$. If $S$ has natural density zero, then the density of $\mathcal{S}$ in $\mathcal{A}$ is zero.
\end{lem}

\begin{proof} Let $a(n)$ denote the number of finite abelian groups of order $n$ and let $B(x) = \sum_{n \le x} a(n)^2$.  By Theorems \ref{thm:erdos_szekeres} and \ref{wk}, there are constants $C$ and $H$ such that $\mathcal{A}(x) \sim Cx$ and $B(x) \sim Hx$. To compute the density of $\mathcal{S}$ in $\mathcal{A}$, consider 
\[ \frac{\mathcal{S}(x)}{\mathcal{A}(x)} \sim \frac{1}{Cx} \mathcal{S}(x) \leq \frac{1}{Cx} \sum_{\substack{n  \le x \\ n \in S} } a(n).\] The inequality follows from the fact that there are at most $a(n)$ elements in $\mathcal{S}$ of order $n$.  We will prove that the limit as $x \to\infty$ of the last expression above is zero. By the Cauchy-Schwarz inequality,

\[ \sum_{\substack{n  \le x \\ n \in S} }a(n)  \le  \left( \sum_{\substack{n  \le x \\ n \in S}} a(n)^2 \right)^{1/2}  \left( \sum_{\substack{n  \le x \\ n \in S} }1^2 \right)^{1/2} \le \left( \sum_{n  \le x} a(n)^2 \right)^{1/2} ( S(x))^{1/2}. \] Dividing by $x$ on both sides and taking limits gives:
\[  \lim_{x \rightarrow \infty} \frac{1}{x}\sum_{\substack{n \in S \\ n  \le x} }a(n)  \le  \lim_{x \rightarrow \infty} \left( \frac{ \sum_{n  \le x} a(n)^2 }{x} \right)^{1/2}  \left( \frac{S(x)}{x} \right)^{1/2}  = (H)^{1/2} (0) = 0.\]
This completes the proof.
\end{proof}

\subsection{The distribution of primes}

To analyze the frequency of certain `building blocks' for the orders of realizable groups, we will use the prime number theorem. 

\begin{thm}[Prime Number Theorem]
Let $\pi(x)$ denote the number of primes $p \le x$.  Then,
\[ \pi(x) \sim  \frac{x}{\log x}. \]
In particular, the natural density of primes is zero. 
\end{thm}

The prime number theorem can be used to compute the natural densities of many other sets built using primes. For instance, it implies that the natural density of prime powers is also zero. This helps us to compute the realizable densities of the finite cyclic groups and of the indecomposable finite abelian groups. Though we invoke the prime number theorem as needed for convenience, it is a much stronger theorem than what we actually require; for our purposes, the bounds on $\pi(x)$ derived by Chebyshev in 1850---that $\pi(x)$ is bounded both above and below by a constant times $x/\log x$ for $x \gg 0$---suffice.

\subsection{The Landau--Ramanujan Theorem}

Sometimes the `building blocks' for orders of realizable groups involve sums of two squares; for this reason, the Landau--Ramanujan Theorem is useful. By a well-known result of Fermat and Euler, a natural number $n$ can be represented as the sum of two squares if and only if every prime factor $p \equiv 3 \pmod 4$ occurring in the prime factorization of $n$ has an even exponent. In 1908, Landau established the precise asymptotic growth for the counting function of such integers (Ramanujan discovered a nearly identical estimate in 1913).

\begin{thm}[\cite{landau1908}]
\label{thm:landau}
Let $L$ denote the set of natural numbers representable as a sum of two squares. Then, 
\[ L(x) \sim K \frac{x}{\sqrt{\log x}} \]
where $K \approx 0.7642$ is the Landau-Ramanujan constant. In particular, the natural density of $L$ is 0.
\end{thm}

\subsection{The Euler totient function}

Some of the building blocks used in the classification of orders of realizable cyclic groups have the form $q-1$ and $q^s(q-1)$. These are explicit values of the Euler totient function $\varphi(n)$. Erd\H{o}s proved that the natural density of the image of $\varphi$ is zero; this follows from the first statement in the next theorem. An excellent discussion of the distribution of the elements in $\im \varphi$ may be found in \cite{ford}.

\begin{thm}[\cite{erdos35}] \label{thm:erdos_totient}
 Given $0 < \epsilon < 1$, there is a constant $C_\epsilon$ such that \[(\im \varphi)(x) \leq C_\epsilon \frac{x}{(\log x)^{1-\epsilon}}\] for $x \gg 0.$ Consequently, for any $0 < \epsilon < 1,$ there is a constant $D_\epsilon$ such that \[ \sum_{\substack{v \in \im \varphi \\ v \leq x}} \frac{1}{v} \leq D_\epsilon (\log x)^\epsilon\] for $x \gg 0.$
\end{thm}
\begin{proof}
    The first part is what appears in \cite{erdos35}; for the second part, apply Abel's summation formula to obtain \[ \sum_{\substack{v \in \im \varphi \\ v \leq x}} \frac{1}{v} = \frac{(\im \varphi)(x)}{x} + \int_1^x \frac{(\im \varphi)(t)}{t^2} \, dt.\] For $x \gg 0$ and some constant $C_\epsilon,$ the first term is bounded by \[ \frac{C_\epsilon}{(\log x)^{1-\epsilon}} = C_\epsilon (\log x)^{\epsilon -1} \leq C_\epsilon(\log x)^\epsilon.\] For the second term, we may ignore the part of the integral over $[1,2].$ Using a simple $u$-substitution ($u = \log t$), \[\int_2^x \frac{(\im \varphi)(t)}{t^2} \, dt \leq C_\epsilon \int_2^x \frac{dt}{t(\log t)^{1-\epsilon}} =(C_\epsilon/\epsilon)((\log x)^\epsilon - (\log 2)^\epsilon) \leq K_\epsilon (\log x)^\epsilon\] for $x \gg0$ and some constant $K_\epsilon.$ Combining the above bounds, the second statement of the theorem now follows.
\end{proof}

The following consequence of Theorems \ref{thm:landau} and \ref{thm:erdos_totient} is a key tool used to compute the realizable density of the finite abelian groups.

\begin{lem} \label{imphiL}
    Let $L\subseteq \N$ denote the set of sums of two squares. The natural density of the product set $S = (\im \varphi)L$ is zero.
\end{lem}
\begin{proof}
    Since $\im \varphi$ has density zero, we can remove 1 from $L$ so that $L(x) = 0$ for $x < 2.$ Theorem \ref{thm:landau} implies that there is a constant $C$ such that $L(x) \leq Cx/\sqrt{\log x}$ for all $x > 1$. We have \[
    \begin{aligned}
        S(x) \leq \sum_{\substack{v \in \im \varphi \\ v \leq x}} L(x/v) = \sum_{\substack{v \in \im \varphi \\ v \leq x/2}} L(x/v) \leq \sum_{\substack{v \in \im \varphi \\ v \leq x/2}} \frac{Cx/v}{\sqrt{\log x/v}} = Cx \underbrace{\sum_{\substack{v \in \im \varphi \\ v \leq x/2}} \frac{1}{v\sqrt{\log x/v}}}_{V(x)}.\\
    \end{aligned}
    \] Now, to prove that $S(x)/x \to 0$ as $x \to \infty,$ it suffices to prove that $V(x) \to 0$ as $x \to \infty.$

    Assume $x$ is large and let $M = \log x$. We will break $V(x)$ into pieces; the first piece will include terms with $v \leq \frac{x}{2e^{\sqrt{M}}}.$ For such $v$, $2e^{\sqrt{M}} \leq x/v$ and so $\log x/v \geq \log 2e^{\sqrt{M}}.$ The sum of the terms in $V(x)$ with $v \leq \frac{x}{2e^{\sqrt{M}}}$ is therefore bounded above by \[\frac{1}{\sqrt{\log{2e^{\sqrt{M}}}}} \sum_{\substack{v \in \im \varphi \\ v \leq x}}\frac{1}{v}.\] For $x$ sufficiently large and any $\epsilon > 0$, the summation in the expression above is bounded by $C'(\log x)^\epsilon$ for some constant $C'$ (see Theorem \ref{thm:erdos_totient}), and hence \[ \frac{1}{\sqrt{\log{2e^{\sqrt{M}}}}} \sum_{\substack{v \in \im \varphi \\ v \leq x}}\frac{1}{v} \leq \frac{C'(\log e^M)^\epsilon}{\sqrt{\log{2e^{\sqrt{M}}}}}\leq \frac{C'M^\epsilon}{\sqrt{\log{2}+\sqrt{M}}} \leq C'M^{\epsilon - 1/4}. \] The last expression above goes to zero as $x \to \infty$ if we take $0 < \epsilon < 1/4$. 

    Next consider intervals $\frac{x}{2e^{i+1}} \leq v \leq \frac{x}{2e^i}$. The remaining terms in $V(x)$ are covered by the $v$ in these intervals for integers $i$ with  $0 \leq i \leq \sqrt{M}.$ For any $0 < \epsilon < 1,$ there is a constant $C''$ such that $(\im \varphi)(x) \leq C'' x/(\log x)^\epsilon$ (again, see Theorem \ref{thm:erdos_totient}). For a fixed $i$ we have \[
    \begin{aligned}
    \sum_{\frac{x}{2e^{i+1}} \leq v \leq \frac{x}{2e^i}}\frac{1}{v\sqrt{\log x/v}} &\leq \frac{C''x/(2e^i)}{(\log (x/(2e^i)))^\epsilon} \frac{1}{(x/2e^{i+1})(\sqrt{\log (2e^i)})} \\ &\leq \frac{C''e}{(M - i - \log 2)^\epsilon(\sqrt{i + \log 2})} \\
    & \leq \frac{C''e}{(M - \sqrt{M} - \log 2)^\epsilon(\sqrt{\log 2})}.
    \end{aligned}\] Summing over all intervals where $0 \leq i \leq \sqrt{M},$ we obtain for some constant $C'''$ an upper bound of \[ C''' \frac{1 + \sqrt{M}}{(M - \sqrt{M} - \log 2)^\epsilon}. \] Since this quantity goes to zero as $x \to \infty$ if we pick $1/2 < \epsilon < 1$, the proof is complete.
\end{proof}

\subsection{Integer partitions}
For any natural number $n$, let $p(n)$ denote the number of unordered partitions of $n$ into positive integers. The function $p(n)$ has rich number-theoretic properties and its asymptotic growth was studied by Hardy and Ramanujan.

\begin{thm}[\cite{HardyRamanujan1918}]
\label{thm:hardy_ramanujan}
Let $p(n)$ denote the number of integer partitions of $n$. Then,
\[ p(n) \sim \frac{1}{4n\sqrt{3}} \exp\left( \pi \sqrt{\frac{2n}{3}} \right). \]
\end{thm}
Because $p(n)$ exhibits sub-exponential growth, any structure that is enumerated by partitions of a logarithmic bound (e.g., partitions of $\log_2 x$) will have a counting function bounded by $\exp(C \sqrt{\log x})$, making them sparse. We use this theorem to prove that the realizable density of the finite abelian 2-groups is 1.

\subsection{Stolz--Ces\`aro Theorem}
When computing the realizable densities of finite $p$-groups, we make use of the following discrete analog of L'H\^opital's rule that is not so well known.

\begin{thm}[\cite{choudary2014real}]
Let $(a_n)$ and $(b_n)$ be sequences of real numbers with
$b_n>0$ for all sufficiently large $n$, and suppose that
$\sum_{k=1}^n b_k \to \infty .$  Then,
\[
\lim_{n\to\infty}
\frac{\sum_{k=1}^n a_k}{\sum_{k=1}^n b_k}
=
\lim_{n\to\infty}\frac{a_n}{b_n},
\]
provided $\lim_{n \rightarrow \infty} a_n/b_n$ exists.
\end{thm}

\subsection{Products and unions involving sets of density zero}

If we have two sets $A$ and $B$ of natural numbers each with natural density zero, then the product set \[AB = \{ab \mid a \in A, b \in B\}\] need not have natural density zero. For example, if $X$ is the set of natural numbers whose prime divisors are 1 mod 4 and if $Y$ is the set of natural numbers whose prime divisors are 2 or 3 mod 4, then $\nd(X) = \nd(Y) = 0$, $XY \cup\{1\} = \N$,  and $\nd(XY) = 1$. Note that the sum of the reciprocals of the elements in each of $X$ and $Y$ diverges. The next lemma shows that if even one of the two zero-density sets $A$ and $B$ has a convergent reciprocal sum, then we can conclude that $AB$ has density zero. In particular, if the sum of the reciprocals of elements in a subset of $\N$ converges, then the set has density zero.

\begin{lem}
\label{lem:product_density}
Let $A$ and $B$ be subsets of the natural numbers. Suppose that the natural density of $B$ is $0$ and that the sum of the reciprocals of the elements in $A$ converges; that is, $\sum_{a \in A} \frac{1}{a} < \infty$. Then the natural density of the product set $AB = \{ab \mid a \in A, b \in B\}$ is $0$.
\end{lem}

\begin{proof}
Let $C = AB$. Enumerate the elements in $A$: $A = \{a_1, a_2, \dots\}.$ To compute $C(x)$, we observe that for each fixed $a \in A$, the number of elements $b \in B$ such that $ab \le x$ is exactly $B(x/a)$. Summing this quantity over all possible choices of $a \in A$ provides an upper bound for $C(x)$. Thus, for any fixed positive integer $M$,
\[ \begin{aligned} \frac{C(x)}{x} &\leq \sum_{k=1}^\infty \frac{B(x/a_k)}{x} \\&= \sum_{k=1}^M \frac{B(x/a_k)}{x} + \sum_{k=M+1}^\infty \frac{B(x/a_k)}{x} \\&\leq \sum_{k=1}^M \frac{1}{a_k} \left(\frac{B(x/a_k)}{x/a_k}\right) + \sum_{k=M+1}^\infty \frac{1}{a_k}.\end{aligned}\] Since $\nd(B) = 0$, the limit as $x \to \infty$ of the first summation just above is zero, and thus \[ \nd^+(C) \leq \sum_{k=M+1}^\infty \frac{1}{a_k}\] for every positive integer $M$. Since the series $\sum_{k=1}^\infty 1/a_k$ converges, $\nd^+(C) = 0$ is forced, from which it follows that $\nd(C) = 0$.
\end{proof}

Finally, we make the simple observation that the natural density of set is unaffected by the inclusion or exclusion of a zero-density subset.

\begin{lem} \label{uniondensity}
Let $A$ and $B$ be subsets of the natural numbers. If the natural density of $A$ is $0$ and the natural density of $B$ exists, then the natural density of $A \cup B$ exists and is equal to the natural density of $B$.
\end{lem}

\begin{proof} It is clear that \[ B(x) \leq (A\cup B)(x) \leq A(x) + B(x).\] Dividing by $x$, taking the limit as $x \to \infty$, and using the fact that $\nd(B)$ exists, we obtain $\nd(B) = \nd(A\cup B)$ by the squeeze theorem.
\end{proof}

\section{Realizable density of finite cyclic groups}
In \cite{PearsonSchneider1970} it is proved that a finite cyclic group $\C_n$ is the group of units of some ring if and only if its order is a product of pairwise coprime integers, each drawn from the set $R_0 = R_1 \cup R_2 \cup R_3 \cup R_4$, where the $R_i$ are defined as follows:
\begin{align*}
R_1 &= \{ q^t - 1 \mid q \text{ is prime, } t \ge 1 \} \\
R_2 &= \{ q^s(q-1) \mid q \text{ is an odd prime, } s \ge 1 \} \\
R_3 &= \{ 4m + 2 \mid m \ge 0 \} \\
R_4 &= \{ 4k \mid k \text{ is an odd positive integer, and } p \mid k \implies p \equiv 1 \pmod 4 \}.
\end{align*}

\n Let $R$ denote the set of all such realizable orders $n$. Since the number of isomorphism classes of cyclic groups of order at most $n$ is just $n$, and since there is a unique cyclic group of any given order, the realizable density of $\mathcal{Z}$ is exactly the natural density of $R$. We can immediately observe that $R_4$ has natural density zero by Landau's theorem (Theorem \ref{thm:landau}) because every integer in $R_4$ is representable as a sum of two squares, and $R_2$ has natural density zero since it is a subset of the image of $\varphi$ (see Theorem \ref{thm:erdos_totient}). To prove that $\delta(\mathcal{Z}) = 1/4$ we will first show that the natural density of $R_1$ is zero and the natural density of $R_3$ is $1/4$; then, we will show how to combine these results to finish the computation of $\delta(\mathcal{Z}).$

\begin{lem} \label{ppd}
The natural density of $R_1 = \{ q^t - 1 \mid q \text{ is prime, } t \ge 1 \}$ is $0$.
\end{lem}

\begin{proof}
Since natural density is translation invariant and a subset of a zero-density subset has density zero, it is enough to show that the density of all prime powers is zero. To that end, let $P$ denote the set of all prime powers. Since $p^t \leq x \iff p \leq x^{1/t},$ \[P(x) \leq \sum_{t=1}^\infty \pi(x^{1/t}) = \sum_{t=1}^{\lfloor \log_2 x \rfloor} \pi(x^{1/t}),\] where the equality above follows from the fact that $\pi(x^{1/t}) = 0$ is zero when $t > \log_2 x.$ Now,
\[
\begin{aligned}
\frac{P(x)}{x} &\leq \frac{\pi(x)}{x} + \sum_{t=2}^{\lfloor \log_2 x \rfloor} \frac{\pi(x^{1/t})}{x}\\
&\leq \frac{\pi(x)}{x} + \frac{(\log_2 x)x^{1/2}}{x} \\
&= \frac{\pi(x)}{x} + \frac{(\log_2 x)}{\sqrt{x}}
\end{aligned}
\]
Using the prime number theorem, the limit of the expression just above is zero, so $\nd(P) = 0$, and the proof is complete.
\end{proof}

\begin{lem} \label{b3density}
The natural density of $R_3 = \{ 4m + 2 \mid m \ge 0 \}$ is $\frac{1}{4}$.
\end{lem}
\begin{proof}
The set $R_3$ consists of all natural numbers that are congruent to $2 \pmod 4$. Since $1 \leq 4m+2 \leq x$ if and only if $0 \leq m \leq (x - 2)/4,$ we have \[R_3(x) = \left\lfloor \frac{x+2}{4}\right\rfloor.\] Hence, \[\lim_{x\to\infty} \frac{R_3(x)}{x} = \lim_{x\to\infty} \frac{\left\lfloor \frac{x+2}{4}\right\rfloor}{x}= 1/4,\] so $\nd(R_3) = 1/4.$
\end{proof}

\begin{lem} \label{lem:odd_sparse}
Let $O$ be the set of all integers formed by multiplying pairwise coprime elements of $\{2^t - 1 \mid t \ge 1\}$.  Then,\[\sum_{o \in O} \frac{1}{o} < \infty.\] In particular, the natural density of $O$ is 0.
\end{lem}
\begin{proof}
Other than 1, each element in $O$ has the form $\prod_{i=1}^k (2^{1+ t_i} - 1)$ where the integers $t_i \ge 1$ are distinct since the factors are pairwise coprime. Let $t = \sum_{i=1}^kt_i$. Then, \[2^t \leq \prod_{i=1}^k (2^{1+ t_i} - 1).\] Let $p'(n)$ denote the number of unordered partitions of $n$ into distinct parts. The inequality above implies that (set $p'(0) = 1$) \[ \sum_{o \in O} \frac{1}{o} \leq \sum_{t=0}^\infty \frac{p'(t)}{2^t} = \prod_{m=1}^\infty (1 + (1/2)^m).\] The product converges since $\sum_{m=1}^\infty (1/2)^m$ converges. Since the sum of the reciprocals of the elements in $O$ converges, $\nd(O) = 0.$
\end{proof}

We now have all of the tools needed to solve our problem for cyclic groups. 

\begin{thm}
The realizable density of the class $\mathcal{Z}$ of finite cyclic groups is  $\delta(\mathcal{Z}) = 1/4$.
\end{thm}

\begin{proof}

As above, let $R$ denote the set of orders of realizable finite cyclic groups. Then, $n \in R$ if and only if $n$ is a product of pairwise coprime integers belonging to the set $R_0 = R_1^o \cup R_1^e \cup R_2 \cup R_3 \cup R_4$, where:
\begin{align*}
R_1^o &= \{ 2^t - 1 \mid t \ge 1 \} \\
R_1^e &= \{ q^t - 1 \mid q \text{ is an odd prime, } t \ge 1 \} \\
R_2 &= \{ q^s(q-1) \mid q \text{ is an odd prime, } s \ge 1 \} \\
R_3 &= \{ 4m + 2 \mid m \ge 0 \} \\
R_4 &= \{ 4k \mid k \text{ is an odd positive integer, and } p \mid k \implies p \equiv 1 \pmod 4 \}.
\end{align*}

Write $n = n_1 \cdots n_k$ where each $n_i \in R_0$ and the $n_i$ are pairwise relatively prime. The only odd elements in $R_0$ are the elements in $R_1^o$. At most one of the $n_i$ can be even. Thus, $n$ must factor as $n=o$ or $n = eo,$ where $e \in R_0\setminus R_1^o$ and $o \in O$ (that is, $o$ is a pairwise product of relatively prime integers of the form $2^t-1$). So we can view $R$ is the union of three sets: $O,$ $R_3 O,$ and $(R_1^e \cup R_2 \cup R_4)O.$ We already proved that the first set has density zero in Lemma \ref{lem:odd_sparse}. The last set has density zero by Lemmas \ref{lem:product_density}, \ref{uniondensity}, \ref{ppd}, and \ref{lem:odd_sparse}. Since $R_3$ is exactly the set of integers of the form $2k$ where $k$ is odd and $1 \in O$, we have $R_3O = R_3,$ so the natural density of the second set is 1/4 by Lemma \ref{b3density}. Now, by Lemma \ref{uniondensity}, $\nd(R) = 1/4$ and the proof is complete.
\end{proof}

\section{Realizable density of indecomposable finite abelian groups}

Recall that the indecomposable finite abelian groups are the cyclic groups of prime power order. It was shown in \cite{CLabelian} that the realizable indecomposable finite abelian groups are exactly the groups $\C_2, \C_4, \C_8$, $\C_p$ for $p$ a Mersenne prime, and $\C_{q-1}$ for $q$ a Fermat prime. 

\begin{thm}
The realizable density of indecomposable finite abelian groups is $0$.
\end{thm}

\begin{proof}
In computing the realizable density of any infinite class, we may ignore any finite subset of the realizable elements in that class without affecting the density. So, we will let $\mathcal{I}_r$ denote the collection of groups $\C_k$ when $k$ is a Mersenne prime or $k+1$ is a Fermat prime. The number Mersenne primes $k$ at most $x$ and the number of Fermat primes $k+1$ at most $x$ are each bounded above by the number of powers of 2 less than $x+1$. Thus, \[ \mathcal{I}_r(x) \leq 2\log_2(x + 1). \]

A finite abelian group is indecomposable if and only if it is cyclic of prime power order. So, we certainly have $\mathcal{I}(x) \geq \pi(x)$. Thus, \[ \frac{\mathcal{I}_r(x)}{\mathcal{I}(x)} \leq \frac{2 \log_2(x+1)}{\pi(x)} \sim \frac{2(\log_2(x+1))(\log x)}{x}\]
by the prime number theorem. The limit of the last expression above is zero, so the realizable density of $\mathcal{I}$ is zero.
\end{proof}

\section{Realizable density of abelian $p$-Groups}
In this section, we determine the realizable density of the class $\mathcal{A}_p$ of finite abelian $p$-groups.  Let $\mathcal{A}_{p,n}$ denote the set of isomorphism classes of abelian $p$-groups of order $p^n$. It is a well-known consequence of the fundamental theorem of finitely generated abelian groups that $|\mathcal{A}_{p,n}| = p(n)$, where $p(n)$ is the integer partition function. Let $\mathcal{R}_{p,n}$ denote the subset of $\mathcal{A}_{p,n}$ consisting of realizable groups. 

By definition,
\[
\delta(\mathcal{A}_p) = \lim_{N \to \infty} \frac{\sum_{n=1}^N |\mathcal{R}_{p,n}|}{\sum_{n=1}^N |\mathcal{A}_{p,n}|}.
\]
Because the sequence of denominators $\sum_{n=1}^N |\mathcal{A}_{p,n}| = \sum_{n=1}^N p(n)$ is strictly increasing and diverges to infinity, we may attempt to apply the Stolz-Ces\`{a}ro theorem by computing the limit
\[
\lim_{n \to \infty} \frac{|\mathcal{R}_{p,n}|}{|\mathcal{A}_{p,n}|} = \lim_{n \to \infty} \frac{|\mathcal{R}_{p,n}|}{p(n)}
\]

\begin{thm} \label{p-groups}
Let $p$ be a prime. The realizable density of finite abelian $p$-groups is $1$ if $p = 2$ and $0$ if $p$ is odd.
\end{thm}

\begin{proof}
We divide the proof into two cases based on the parity of the prime $p$.

If $p$ is odd, then from \cite{Ditor1971} we know that if a finite abelian $p$-group is realizable, then it must be elementary abelian; since there is a unique such group of order $p^n$, we know that $|\mathcal{R}_{p, n}|$ is either 1 or 0 for all $n$ (depending on whether $\C_p$ is realizable, which is true if and only if $p$ is a Mersenne prime). In either case, \[
\lim_{n \to \infty} \frac{|\mathcal{R}_{p,n}|}{|\mathcal{A}_{p,n}|} \leq \lim_{n \to \infty} \frac{1}{p(n)} = 0,
\]
and so by the Stolz-Ces\`{a}ro theorem we have that the realizable density is zero.

Now consider the prime $p = 2.$ In \cite{CLpgroups} it is shown that for any commutative ring with unity $R$ and any $R$-module $M$, the group $R^\times \times (M, +)$ is realizable as the group of units in a ring with characteristic $\mathrm{char}\, R$. Hence, since any abelian group is a $\Z$-module and $\Z^\times = \C_2,$ we have that $\C_2 \times A$ is realizable as the group of units in a ring of characteristic zero for any finite abelian $2$-group $A$. Therefore to establish a lower bound on $|\mathcal{R}_{2,n}|$, it suffices to count the number of abelian $2$-groups of order $2^n$ that contain at least one $\C_2$ direct factor.

Under the bijection between abelian $2$-groups of order $2^n$ and the integer partitions of $n$, a group with a $\C_2$ factor corresponds to a partition of $n$ that contains at least one part of size $1$. We can count the exact number of such partitions by establishing a bijection: removing a single part of size $1$ from a partition of $n$ yields an unrestricted partition of $n-1$. Thus, there are exactly $p(n-1)$ partitions of $n$ containing at least one $1$.

Because every group corresponding to these $p(n-1)$ partitions is realizable, we have $|\mathcal{R}_{2,n}| \geq p(n-1)$. Thus,
\[
\frac{|\mathcal{R}_{2,n}|}{|\mathcal{A}_{2,n}|} \geq \frac{p(n-1)}{p(n)}.
\]
To evaluate $\lim_{n \to \infty} \frac{p(n-1)}{p(n)}$ we apply the Hardy-Ramanujan asymptotic formula for the partition function, $p(n) \sim \frac{1}{4n\sqrt{3}} \exp(\pi \sqrt{2n/3})$:
\[
\frac{p(n-1)}{p(n)} \sim \frac{n}{n-1} \exp\left(\pi \sqrt{\frac{2}{3}} \left(\sqrt{n-1} - \sqrt{n}\right)\right)
\]

By rationalizing the exponent, we observe that $\sqrt{n-1} - \sqrt{n} = \frac{-1}{\sqrt{n-1} + \sqrt{n}}$. As $n \to \infty$, this term vanishes. Consequently, the exponential term approaches $e^0 = 1$, and the rational factor $\frac{n}{n-1}$ also approaches $1$. Thus:
\[
\lim_{n \to \infty} \frac{p(n-1)}{p(n)} = 1.
\]
Since $\lim_{n \to \infty} \frac{|\mathcal{R}_{2,n}|}{|\mathcal{A}_{2,n}|}$ cannot exceed $1$, we conclude by the Stolz-Ces\`{a}ro theorem that the cumulative asymptotic density of realizable abelian $2$-groups is exactly $1$.
\end{proof}

This theorem highlights a stark dichotomy between the prime $2$ and all odd primes: asymptotically, almost every finite abelian $2$-group is realizable, whereas almost no odd $p$-group is realizable.

\section{Realizable density of finite abelian groups}

First we will prove that the density in $\mathcal{A}$ of finite abelian groups realizable via a ring of positive characteristic is zero; then we will prove that the density in $\mathcal{A}$ of finite abelian groups realizable via a ring of characteristic zero is 1/2. This will establish $\delta(\mathcal{A}) = 1/2.$

\subsection{Density in positive characteristic}

If an abelian group $G$ is the group of units of a ring of positive characteristic, then $G$ is the group of units in a finite commutative ring $R$. By the structure theorem for commutative Artinian rings, $R \cong \prod_{i=1}^k R_i$, where each $R_i$ is a finite local ring with residue field of prime power order $p_i^{f_i}$. Their unit groups are $R_i^\times \cong \C_{p_i^{f_i}-1} \times P_i$, where $P_i$ is a $p_i$-group. This means $|G| = \prod_{i=1}^k (p_i^{f_i} - 1)p_i^{a_i}$ where $k \geq 1, f_i \geq 1$ and $a_i \geq 0$. The primes $p_i$ need not be distinct, but we can regroup the factors in this product by combining factors where the primes $p_i$ are the same; thus $|G|$ is a product over a set of \textit{distinct} primes $p$ of terms of the form \[ p^a\prod_{j=1}^r (p^{f_j} - 1)\] where $r\geq 1, f_j \geq 1,$ and $a \geq 0$. Note that when $r = 1$ and $f_1 =1$, this term is $\varphi(p^{a+1})$. Thus, it is helpful to consider these factors separately since the product of these factors lies in the image of Euler's $\varphi$-function (because the primes $p$ are distinct), which has natural density zero. If we exclude terms where $r = 1$ and $f_1 = 1$, any remaining term that has $f_j = 1$ for some $j$ must also have $r \geq 2.$ So, each remaining factor is a product of numbers taken from the set $G = G_1 \cup G_2 \cup G_3$ where:
\begin{itemize}[leftmargin=2cm]
    \item[$G_1=$] $\{(p^f - 1)p^a \, | \, \text{ $f \geq 2, a \geq 0$}\}$
    \item[$G_2=$] $\{(p-1)^tp^a \, | \, \text{ $t \geq 2, a \geq 0$}\}$
    \item[$G_3=$] $\{(p-1)(p^f - 1)p^a \, | \, \text{ $f \geq 1, a \geq 0$}\}.$
\end{itemize}

\begin{lem}\label{lem:reciprocal_sums}
Define $G$ as above and let $B$ be the set of all finite products of the elements in $G$. The sum of the reciprocals of the elements in $B$ converges.
\end{lem}
\begin{proof}
First, we prove that the sum of the reciprocals of elements in $G$ converges. To do so, it suffices to prove that the sum of the reciprocals of the elements in each $G_i$ converges for $i = 1, 2, 3.$

For $G_1$:
\[ \sum_{g \in G_1} \frac{1}{g} = \sum_{p} \sum_{f \ge 2} \sum_{a=0}^{\infty} \frac{1}{(p^f - 1)p^a} \le \sum_{p} \sum_{f \ge 2} \frac{1}{p^f/2} \left( \frac{p}{p-1} \right) \le \sum_{p} \frac{4}{p^2} < \infty. \]

For $G_2$, we can ignore the elements with $p = 2$ ($\{ 2^a \, | \, a \geq 0\}$) because the sum of the reciprocals of such elements converges. For the remaining elements, we have:
\[ 
\begin{aligned}
\sum_{p>2} \sum_{t \ge 2} \sum_{a=0}^{\infty} \frac{1}{(p-1)^tp^a} &= \sum_{p>2} \sum_{t \ge 2} \frac{p}{(p-1)^{t+1}} \leq \sum_{p>2} \sum_{t \ge 2} \frac{p2^{t+1}}{p^{t+1}}\\
& = \sum_{p>2} \frac{8}{p^2} \sum_{t \ge 0} \left(\frac{2}{p}\right)^t= \sum_{p>2}\frac{8}{p(p-2)}< \infty.
\end{aligned}\] 

For $G_3$:
\[ 
\begin{aligned}
\sum_{g \in G_3} \frac{1}{g} &= \sum_{p} \sum_{f \ge 1} \sum_{a=0}^{\infty} \frac{1}{(p-1)(p^f - 1)p^a} = \sum_{p} \sum_{f \ge 1} \frac{p}{(p-1)^2(p^f - 1)}\\
& \leq \sum_{p} \sum_{f \ge 1} \frac{2p}{(p-1)^2p^f} = \sum_{p} \sum_{f \ge 0} \frac{2}{(p-1)^2p^f} = \sum_{p}\frac{2p}{(p-1)^3} \leq \sum_{p} \frac{16}{p^2} < \infty.
\end{aligned}\]
Taken together, these facts imply that $\sum_{g\in G} \frac{1}{g}$ converges.

Because $\sum_{g \in G} \frac{1}{g}$ converges and $g \ge 2$ for all $g \in G\setminus\{1\}$, the series $\sum_{g \in G\setminus\{1\}} \frac{1}{g - 1}$ also converges by the limit comparison test. By the standard properties of infinite products, the convergence of $\sum \frac{1}{g - 1}$ guarantees the convergence of the infinite Euler product $\prod_{g \in G\setminus\{1\}} \left(1 + \frac{1}{g-1}\right)$. Expanding each term in this product as a geometric series yields $\prod_{g \in G\setminus\{1\}} \left( 1 + \frac{1}{g} + \frac{1}{g^2} + \dots \right)$. This expansion generates the reciprocal of every element in $B$. Therefore:
\[ \sum_{b \in B} \frac{1}{b} \le \prod_{g \in G\setminus\{1\}} \left( 1 + \frac{1}{g - 1} \right) < \infty \]
Thus, the sum of the reciprocals of $B$ converges.
\end{proof}

\begin{thm} \label{lem:finite_char}
Let $\mathcal{P}$ denote the set of isomorphism classes of finite abelian groups that occur as the group of units in a ring of positive characteristic. The density of $\mathcal{P}$ in $\mathcal{A}$ is zero. 
\end{thm}
\begin{proof} By Lemma \ref{reductiontonaturaldensity}, it suffices to prove that the set $P$ of orders of groups in $\mathcal{P}$ has natural density zero. We explained before the proof of Lemma \ref{lem:reciprocal_sums} that each $n \in P$ can be factored as $n = ab$ where $a$ is in the image of the Euler $\varphi$-function and $b$ lies in the set $B$ defined in Lemma \ref{lem:reciprocal_sums}. Thus $P$ is contained in the product set $(\im \varphi)B$. Since $\im \varphi$ has natural density zero and the sum of the reciprocals of the elements in $B$ converges, by Lemma \ref{lem:product_density} their product set has density $0$. Since $P$ is contained in this product set, the density of $P$ is $0$.
\end{proof}

\subsection{Density in arbitrary characteristic}

We first prove a lemma that will give us a lower bound on the realizable density of $\mathcal{A}.$

\begin{lem} \label{lem:C_2factor}
    For any finite abelian group $H$, $\C_2 \times H$ is realizable as the group of units in a ring of characteristic zero; the density of such groups in $\mathcal{A}$ is 1/2.
\end{lem}

\begin{proof}
Let $\mathcal{Q}$ denote the collection of isomorphism classes of finite abelian groups of the form $\C_2 \times H$. We explained in the proof of Theorem \ref{p-groups} why these groups are realizable via rings of characteristic zero, so we just need to compute the density of this family within $\mathcal{A}$. That is, we need to compute \[ \lim_{x \to \infty} \frac{\mathcal{Q}(x)}{\mathcal{A}(x)}.\] For any integer $n$, let $a(n)$ denote the number of isomorphism classes of abelian groups of order $n$. By the structure theorem for finite abelian groups, the map $H \mapsto \C_2 \times H$ induces a bijection from the groups in $\mathcal{A}$ of order $k$ to the groups in $\mathcal{Q}$ of order $2k$. Thus, the number of groups in $\mathcal{Q}$ of order $n$ is zero if $n$ is odd and $a(n/2)$ if $n$ is even. From this it easily follows that $\mathcal{Q}(x) = \mathcal{A}(x/2)$ for any real number $x$. By Theorem \ref{thm:erdos_szekeres}, there is a constant $C$ where $\mathcal{A}(x) \sim Cx$; thus \[\frac{\mathcal{Q}(x)}{\mathcal{A}(x)} = \frac{\mathcal{A}(x/2)}{\mathcal{A}(x)} \sim \frac{Cx/2}{Cx} = 1/2.\] Hence, \[ \lim_{x \to \infty} \frac{\mathcal{Q}(x)}{\mathcal{A}(x)} = \frac{1}{2},\] and the proof is complete.
\end{proof}

\begin{thm}
The realizable density of $\mathcal{A}$ is $1/2$.
\end{thm}

\begin{proof} 

By Lemma \ref{lem:finite_char}, the density in $\mathcal{A}$ of the finite abelian groups that are realizable as the group of units in a ring of positive characteristic is $0$. Thus, $\delta(\mathcal{A})$ (if it exists) is the density in $\mathcal{A}$ of the class $\mathcal{A}_0$ of finite abelian groups that are realizable as the group of units in a ring of characteristic zero.

In \cite{CD18a} it is proved that if a finite abelian group $G$ is the group of units in a ring of characteristic $0$, then $G$ must contain either $\C_2$ or $\C_4$ as a direct factor. We can therefore write $\mathcal{A}_0$ as the disjoint union of two sets: $\mathcal{Q}$, the finite abelian groups in $\mathcal{A}_0$ with $\C_2$ as a direct factor, and $\mathcal{F}$, the finite abelian groups in $\mathcal{A}_0$ with $\C_4$ as a direct factor but not $\C_2$. In Lemma \ref{lem:C_2factor}, we proved that the density of $\mathcal{Q}$ in $\mathcal{A}$ is 1/2, so $1/2 \leq \delta^-(\mathcal{A})$.

To complete the proof, we must prove that the density of $\mathcal{F}$ in $\mathcal{A}$ is zero. By Lemma \ref{reductiontonaturaldensity}, it suffices to show that the set $F$ of orders of elements of $\mathcal{F}$ has natural density zero. To that end, take $G \in \mathcal{F}$. Then, there is a commutative ring $X$ of characteristic zero, generated by its units, such that $G = X^\times.$ By \cite[Proposition 1]{PearsonSchneider1970}, we may write $X = T \times R$ where $T$ is a finite ring and the nilradical $N$ of $R$ contains the torsion ideal of $R$. We have a short exact sequence
\[ 1 \to 1 + N \to R^\times \to (R/N)^\times \to 1.\] where $R/N$ is a torsion-free ring. So, $|G| = |T^\times||(R/N)^\times||N|.$ Let's examine the three latter factors of $|G|$ in turn.

By \cite[Theorem B]{CD18a}, $|(R/N)^\times| = 2^a 3^b$ where $a, b \geq 0$. Hence, $|(R/N)^\times|$ is either a sum of two squares (if $b$ is even) or 3 times a sum of two squares (if $b$ is odd).

Let's next consider $|N|$. Since $X^\times$ does not have $\C_2$ as a direct factor but it does have $\C_4$ as a direct factor, the element $-1 \in X$ must have a square root in $X$; this means the characteristic zero ring $X$, and hence $R$, is a module over the Gaussian integers $\Z[i]$. Since $R^\times$ is finite, so is $N$, and thus $N$ is a finite module over the Gaussian integers. As such, $N$ is a product of modules of the form $\Z[i]/(a)$ where $a \neq 0$. It is well known that the order of this module is the Gaussian norm of $a$, which is a sum of two squares, and since the Gaussian norm is multiplicative the order of $N$ is a sum of two squares as well.

Finally, consider $|T^\times|$. Since $T$ is finite, we earlier explained that its unit group is a product of groups of the form $\C_{p^f - 1} \times P$ where $P$ is a finite $p$-group. So, $|T^\times|$ is a product of numbers of the form 
\begin{itemize}
    \item $(p^f-1)p^a$ with $f \geq 2,$ or
    \item $2^a,$ or
    \item $(p-1)p^a$ where $p$ is congruent to 1 modulo 4.
    \end{itemize}
In the last case, we are able to exclude primes congruent to 3 modulo 4 because we can't have $p - 1 = 2k$ where $k$ is odd since then the unit group of $T$ would have $\C_2$ as a direct factor coming from $\C_{p-1} = \C_2 \times \C_k$. Note further that in the last case, since $p$ is 1 mod 4, $p^a$ is a sum of two squares.

To summarize: $|G| = 3^\epsilon b x \prod_{i=1}^r(p_i-1)$ where $\epsilon = 0$ or $1$; $b \in B$ ($B$ is the set whose reciprocal sum converges from Lemma \ref{lem:reciprocal_sums}); $x \in L$ ($x$ is a sum of two squares); and the $p_i$ are distinct primes congruent to 1 mod 4. We may assume the primes $p_i$ are distinct since any even power of $(p -1)$ is a sum of two squares and hence can be absorbed in the factor $x$. Defining $S$ as in Lemma \ref{imphiL}, we see that $F \subseteq BS \cup 3BS,$ so $F$ has density zero if $BS$ has density zero. But $BS$ has density zero since $S$ has density zero and the sum of the reciprocals of elements in $B$ converges. This completes the proof.
\end{proof}

\bibliographystyle{plain} 
\bibliography{references}

\end{document}